\documentclass{amsart}
\usepackage[english]{babel}
\usepackage[utf8]{inputenc}
\usepackage{amsmath, amssymb,epic,graphicx,mathrsfs,enumerate}
\usepackage[all]{xy}
\usepackage{xcolor}
\usepackage{comment}
\usepackage{enumitem}
\usepackage{hyperref}
\usepackage[noabbrev,capitalize]{cleveref}
\usepackage[]{frontespizio}
\usepackage{amssymb}
\usepackage{amsmath}
\usepackage{amsthm}
\usepackage{graphicx}
\usepackage{hyperref}
\usepackage{amscd}
\usepackage{enumerate}
\usepackage{mathrsfs}
\usepackage{amsmath}
\usepackage{amsthm}
\usepackage{amssymb}
\usepackage{latexsym}
\usepackage{epsfig}
\DeclareMathOperator{\nil}{\mathcal N}
\DeclareMathOperator{\perm}{Sym}
\DeclareMathOperator{\alt}{Alt}

\DeclareMathOperator{\fit}{Fit}

\DeclareMathOperator{\Z}{Z}

\newcommand{\sym}{\mathrm{Sym}}

\newcommand{\la}{\langle}
\newcommand{\ra}{\rangle}

\newtheorem{thm}{Theorem}
\newtheorem{cor}[thm]{Corollary}
 \newtheorem{lemma}[thm]{Lemma}
\newtheorem{prop}[thm]{Proposition} 
 
\newtheorem{question}[thm]{Question}

\numberwithin{equation}{section}

\renewcommand{\footnote}{\endnote}
\newcommand{\ignore}[1]{}\makeglossary

\title[]{Profinite groups with many elements with large nilpotentizer and generalizations}

\author{Martino Garonzi}
\address{Martino Garonzi. University of Ferrara (Italy), Dipartimento di Matematica e Informatica.
ORCID: https://orcid.org/0000-0003-0041-3131}
\email{martino.garonzi@unife.it}

\author{Andrea Lucchini}
\address{Andrea Lucchini. University of Padova (Italy), Dipartimento di Matematica ``Tullio Levi Civita''. ORCID: https://orcid.org/0000-0002-2134-4991}
\email{lucchini@math.unipd.it}

\author{Nowras Otmen} 
\address{Nowras Otmen. University of Padova (Italy), Dipartimento di Matematica ``Tullio Levi Civita''. ORCID: https://orcid.org/0009-0009-8092-1689}
\email{nowras.naufel@math.unipd.it}

\thanks{Project funded by the EuropeanUnion - NextGenerationEU under the National Recovery and Resilience Plan (NRRP), Mission 4 Component 2 Investment 1.1- Call PRIN 2022 No. 104 of February 2, 2022 of Italian Ministry of University and Research; Project 2022PSTWLB (subject area: PE - Physical Sciences and Engineering) ``Group Theory and Applications''. The second author is member of GNSAGA (INDAM)}

\date{}



\begin{document}
	\bibliographystyle{amsplain}

\begin{abstract}
Given a profinite group $G$ and a family $\mathcal{F}$ of finite groups closed under taking subgroups, direct products and quotients, denote by $\mathcal{F}(G)$ the set of elements $g \in G$ such that $\{x \in G\ |\ \langle g,x \rangle \ \mbox{is a pro-}\mathcal{F} \mbox{ group}\}$ has positive Haar measure. We investigate the properties of $\mathcal{F}(G)$ for various choices of $\mathcal{F}$ and its influence on the structure of $G$.
\end{abstract}

\maketitle

\section{Introduction}
Let $\mathcal F$ be a class of finite groups closed under taking subgroups, quotients and direct products. Given a profinite group $G$ and $x$ in $G$, let $\mathcal F_G(x)$ be the set of the elements $g\in G$ such that $\langle g, x\rangle$ (here and henceforth, 
for any subset $X$ of $G$, writing $\langle X\rangle$ we mean the smallest \emph{closed} subgroup of $G$ generated by $X$) is a pro-$\mathcal F$ group. Let $\mu_G$ be the normalized Haar measure of $G$. By \cite[Section 2]{aemp}, $\mathcal{F}_G(x)$ is closed in $G$, hence measurable, and 
$$\mu_G(\mathcal{F}_G(x)) = \inf_{N \unlhd_{\circ} G} \frac{|\mathcal{F}_{G/N}(xN)|}{|G/N|}.$$

For a positive real number $\varepsilon,$ let $\mathcal F_\varepsilon(G)$ be the set of the elements $g\in G$ such that $\mu_G(\mathcal F_G(g))\geq \varepsilon$. Given an open normal subgroup $N$ of $G$, let
$$\mathcal F_{\varepsilon,N}(G)=\{g\in G\mid \mu_{G/N}(\mathcal F_{G/N}(gN))\geq \varepsilon\}.$$ Notice that $\mathcal F_{\varepsilon,N}(G)$ is closed in $G$ and consequently $\mathcal F_\varepsilon(G)=\cap_N \mathcal F_{\varepsilon,N}(G)$ is closed. But then
$\mathcal F(G)=\cup_{n\in \mathbb N}\mathcal F_{1/n}(G)$ is a Borel subset of $G$ and coincides with the set of the elements $g\in G$ such that $\mu(\mathcal F_G(g))$ is positive.

The previous definitions give rise to the following questions:
\begin{question}\label{q1}Is there a group-theoretical characterization of the elements $g\in G$ with the property that $\mu(\mathcal F_G(g))$ is positive? 
\end{question}

\begin{question}\label{q2}Is $\mathcal F(G)$ a closed  subset of $G$ for every profinite group $G$?
\end{question}

\begin{question}\label{q3}Is $\mathcal F(G)$ a subgroup of $G$ for every profinite group $G$?
\end{question}

The answers to \cref{q2} and \cref{q3} are in general negative. We analyze \cref{q2} in further detail in \cref{families}. For \cref{q3}, if we take a class $\mathcal F$ that doesn't contain a cyclic group of order a power of a prime $p$, then this is already not true. Suppose $C_{p^m} \notin \mathcal F$ and consider $C_q \in \mathcal F$ with $q\neq p$. Let $G=P \rtimes C_q$ be a Frobenius group with $\exp(P)=p^m$. Then $\mathcal F(G)$ is not a subgroup of $G$, since if $x \in G$ has order $q$ then $x^P=xP$ so every element of $P$ is a product of two elements of order $q$.

In the particular case of the class $\mathcal A$ of the finite abelian groups, we can answer all three questions. Since $\mathcal A_G(g)=C_G(g)$ for every $g\in G$, the elements $g$ such that $\mu(\mathcal A_G(g))>0$ are those with finite conjugacy class. Then, it follows that $\mathcal A(G)$ coincides with the FC-center of $G$, which is a (not necessarily closed) subgroup of $G$. If $G=\prod_{i \in I} G_i$ is a direct product of groups, the corresponding restricted direct product is the subset of the direct product consisting of tuples $(g_i)_i$ such that $\{i \in I\ |\ g_i \neq 1\}$ is finite. An example of an FC-center that is not closed is the following: consider $G$ the direct product of infinitely many nonabelian simple groups, in which case the FC-center of $G$ coincides with the restricted direct product of the factors, which is dense but not closed in $G.$ In particular, if $\mu(\mathcal A(G))>0,$ then $\mathcal A(G)$ is an open normal subgroup of $G$ and it follows from \cite[Lemma 2.6]{sh} that the derived subgroup of $\mathcal A(G)$ is finite (which implies that  $\mathcal A(G)$ is virtually abelian). 

For most of the natural examples of classes $\mathcal F$, one can find at least one profinite group $G$ such that the set $\mathcal F(G)$ is not closed. Although we cannot fully characterize the families for which this occurs, the following result is a step in this direction.

\begin{thm}\label{allfingrps}
Let $\mathcal F$ be a class of finite groups closed under taking subgroups, quotients and direct products. Assume that $\mathcal F$ satisfies the following property: for every prime $p$, if $\mathcal F$ contains a finite nontrivial $p$-group, then it contains all the cyclic finite $p$-groups. Then $\mathcal{F}(G)$ is closed in $G$ for every profinite group $G$ if and only if $\mathcal F$ is the class of all finite groups.
\end{thm}

We also show in \cref{exp2} that, if the exponent of the groups in $\mathcal{F}$ is bounded, then we cannot hope to have the above result for $\mathcal{F}$.

\

Recall that the prosolvable radical of a profinite group $G$ is the subgroup of $G$ generated by all the closed normal prosolvable subgroups of $G$. It is denoted by $R(G)$ and it is a closed normal prosolvable subgroup of $G$. The following theorem provides a characterization of $\mathcal F(G)$ for $\mathcal F = \mathcal S$ the class of finite solvable groups.

%and provides, along with \cref{thodd}, a generalization of \cite[Theorem 1.2]{aemp}.

\begin{thm}\label{descs}
Let $\mathcal{S}$ be the class of finite solvable groups. If $G$ is a profinite group, then $\mathcal{S}(G)$ is a subgroup of $G$, $R(G) \leq \mathcal{S}(G)$ and $\mathcal{S}(G)/R(G)$ is the FC-center of $G/R(G)$. Moreover, if $\mathcal S(G)$ is closed, then $\mathcal S(G)/R(G)$ is finite.
\end{thm}

In \cite{aemp}, it is proven that if $\mathcal S(G) = G$, then $G$ is virtually prosolvable. We generalize this result in two different ways, both showing that you don't need \emph{every} element in $G$ to have a positive measure solvabilizer to guarantee virtual prosolvability.

\begin{thm}\label{musolpos} Let $G$ be a profinite group. If $\mu(\mathcal S(G)) > 0$, then $G$ is virtually prosolvable.
\end{thm}

\begin{thm}\label{thodd}
 Let $G$ be a profinite group. If $\mu(\mathcal S_G(g)) > 0$ for every $g$ of odd order (as supernatural number), then $G$ is virtually prosolvable. 
\end{thm}

Let $\mathcal{N}$ be the family of all finite nilpotent groups. Recall that the {(topological)} Fitting subgroup of a profinite group $G$ is the subgroup topologically generated by all the closed normal pronilpotent subgroups of $G$. It is denoted by $\fit(G)$ and it is a closed pronilpotent normal subgroup of $G$.

\begin{thm}\label{nilpp}
Let $G$ be a profinite group and let $g \in G$. If $\mu(\mathcal{N}_G(g)) > 0$, then there is an open subgroup $H$ of $G$ such that $g \in \fit(H)$. In particular, $g$ has finite order modulo $\fit(G)$.
\end{thm}

It follows from \cite[Theorem 1.3 (3)]{aemp} that, if $\mathcal N(G) = G$, then $G$ is virtually pronilpotent, thus its Fitting subgroup is open. The following result is a local version of this.

\begin{thm}\label{localnilp}
	Let $G$ be a  profinite group and let $P$ be a $p$-Sylow subgroup of $G.$  Suppose that $\mu(\nil_G(g))>0$ for every $g\in P.$
	Then the following hold:
	\begin{enumerate}
		\item $N_G(P)$ is an open subgroup of $G.$
		\item $P/O_p(G)$ is finite.
		\item $O_p(G)C_G(g)$ is open in $G$ for every $g \in P$.
		\item If $P$ is finite, then $C_G(P)$ is open in $G$.
	\end{enumerate}
\end{thm}

By \cite[Theorem 1.3]{aemp},  if $\mathcal N(G)=G$, then $G$ is virtually pronilpotent. In \cite{aemp}, it is observed that the condition $\mathcal N(G) = G$ is stronger than $G$ being virtually pronilpotent; indeed, it suffices to consider the example $\mathbb Z_p \rtimes C_2 = \mathbb Z_p \rtimes \la x \ra$, where $p$ is an odd prime and $x$ acts on the $p$-adic integers by inversion, and observe that $\mathcal N_G(x) = \{1,x\}$. It is natural to ask whether, as happens for the class $\mathcal S$ of finite solvable groups, we can reach the conclusion that $G$ is virtually pronilpotent from the  weaker assumption $\mu(\mathcal{N}(G)) > 0.$ The lack of a good characterization of the elements of $\mathcal N(G)$ in terms of the group structure, does not allow us to mimic the argument used in the proof of \cref{thodd} for the  class $\mathcal S$. Therefore, this is an apparently difficult question that we leave open. However, we obtain a partial result, providing an affirmative answer to the analogous question formulated for the class $\mathcal{P}_p(G)$ of finite $p$-groups, for any fixed prime $p.$

\begin{thm}\label{classp}
Let $G$ be a profinite group. If $\mu(\mathcal{P}_p(G)) > 0$ then $G$ is virtually pro-$p$.
\end{thm}

%It is interesting to ask whether $\mu(\mathcal{N}(G)) > 0$ implies that $G$ is virtually pronilpotent. This is a hard question because we do not have a group-theoretical characterization of $\mathcal{N}(G)$. This contrasts with the case of $\mathcal S(G)$; the characterization of elements in $\mathcal S(G)$ conjectured in \cite[Conjecture 1]{solv} and proved in \cite[Theorem 6]{solv} and \cite[Corollary 1.3]{fgg} is fundamental in the proof of \cref{descs}.

If $G$ is a finite group, then the intersection $\cap_{g\in G}\nil_G(g)$ of the nilpotentizers coincides with the hypercenter $Z_\infty(G)$ of $G$ (see \cite[Proposition 2.1]{az}). More generally, when $G$ is a profinite group, $\cap_{g\in G}\nil_G(g)$ is the hypercenter $Z_\infty(G)$, defined as the set of all $x \in G$ such that $xN \in Z_\infty(G/N)$ for every open normal subgroup $N$ of $G$ (see \cite{hyper}). In other words, if $g\in G,$
then $\nil_G(g)=G$ if and only if $g\in Z_\infty(G).$
So it is natural to ask whether \cref{nilpp} can be replaced with a stronger statement: {\sl{Suppose that $g\in G$ is such that $\mu(\nil_G(g))>0.$ Does there exist an open subgroup $H$ of $G$ such that $g\in Z_\infty(G)$?}}
%\textcolor{red}{The same argument used in the proof of \cref{nilpp} demonstrates that an affirmative answer to the previous question would imply that $g$ has finite order modulo $Z_\infty(G).$} \textcolor{blue}{Lo stesso argomento non funziona perché non è vero che, se $M\unlhd_o G$, allora $Z_\infty(M) \leq Z_\infty(G)$ (ad esempio $G=S_3$ e $M=A_3$).} The following result shows that this is not always the case.

\begin{thm}\label{nofiniteorder}
For every positive real number $\varepsilon$ there exists a solvable profinite group $G$ which contains an element $g$ such that {$\mu(\mathcal{N}_G(g)) > 1-\varepsilon$, $g$ has infinite order modulo $Z_{\infty}(G)$ and} there is no open subgroup $H$ of $G$ with $g \in Z_{\infty}(H)$.
\end{thm}

Let $G$ be a profinite group and let $x\in G.$ Consider the set $$\Lambda_G(x)=\{g\in G \mid \langle x, x^g\rangle \text { is pronilpotent}\}.$$ Clearly $\nil_G(x)\subseteq \Lambda_G(x)$, so it is reasonable to ask to what extent the results proven in  Section \ref{nilpot} remain true when substituting $\nil_G(x)$ with $\Lambda_G(x).$ By the Baer-Suzuki  Theorem (cf. \cite[p. 105]{gor}), if $\Lambda_G(x)=G,$ then $x \in \fit(G).$ A natural question is whether the following generalized version of \cref{nilpp} is true: {\sl{if $\mu(\Lambda_G(x))>0,$ then the order of $x$ modulo $\fit(G)$ is finite}.} 
We can state also a $p$-local version of the previous question: {\sl{Assume that $x$ is a $p$-element of $G$ and let $\Lambda_{G,p}(x)=\{g\in G \mid \langle x, x^g\rangle \text { is a pro-$p$ group}\}.$ Does $\mu(\Lambda_{G,p}(x))>0$ imply that the order of $x$ modulo $O_p(G)$ is finite?}}
We prove that the last question has in general a negative answer. This is implied by the following result.

\begin{thm}\label{nobaer}
For every positive real number $\eta<1$, every integer $t \geq 1$
and every prime $p,$ there exists a finite group $G$ such that $O_p(G)=1$ and
$|\Lambda_{G,p}(x)|\geq \eta|G|$ for some element $x \in G$ of order $p^t$.
\end{thm}

Indeed, for every positive integer $t$, choose a finite group $G_t$ with $O_p(G_t)=1$ and containing an element $x_t \in G_t$ of order $p^t$ such that $|\Lambda_{G_t,p}(x_t)| \geq \eta_t |G_t|$ where $\eta_t = 1-1/2^t$. Let $G =\prod_{t \geq 1} G_t$ and $x=(x_t)_t \in G$. It follows that
\[\mu(\Lambda_{G,p}(x)) \geq \prod_{t \geq 1} \eta_t > 0.\]
Moreover $O_p(G)=1$ and $x$ has infinite order.

In \cref{families} we prove \cref{allfingrps}. In \cref{solv} we prove Theorem \ref{descs}, \ref{musolpos} and \ref{thodd}. In \cref{nilpot} we prove Theorems  \ref{nilpp}, \ref{localnilp}, \ref{classp} and \ref{nofiniteorder}. In \cref{Lambda} we prove \cref{nobaer}.

\section{General families of finite groups} \label{families}

In this section we discuss \cref{q2}. In particular we prove Theorem \ref{allfingrps} and, in \cref{exp2}, we give an example to show that the statement of this theorem does not hold if we remove the assumption.

\begin{proof}[Proof of \cref{allfingrps}]

If $\mathcal F$ is the class of all finite groups, then $\mathcal{F}(G)=G$ for every profinite group $G$ and this proves the implication $(\Leftarrow)$.

Now we prove the other implication. Assume that $\mathcal F$ is not the class of all finite groups. We will construct a profinite group $G$ such that $\mathcal{F}(G)$ is not closed.

Assume first that $\mathcal F$ contains all the finite cyclic groups. Since $\mathcal F$ is not the class of all finite groups and every finite group is isomorphic to a subgroup of an alternating group, there exists an alternating group $\alt(m)$ which does not belong to $\mathcal F$. Since $\mathcal{F}$ contains all finite cyclic groups, $\mu(\mathcal{F}_{\alt(m)}(1))=1$. It is known that, for every $x \in \alt(m) \setminus \{1\}$, there exists $y \in \alt(m)$ such that $\langle x,y \rangle = \alt(m)$ \cite{gk}, therefore $0< \mu(\mathcal{F}_{\alt(m)}(x)) < 1$ for all $x \in \alt(m) \setminus \{1\}$. Let $G := \alt(m)^{\mathbb{N}}$ and let $c := \max_{x \in \alt(m)} \mu(\mathcal{F}_{\alt(m)}(x))$. Then $0 < c < 1$. If $g = (g_n)_n \in G$ then
$$\mu(\mathcal{F}_G(g)) = \prod_{n \in \mathbb{N}} \mu(\mathcal{F}_{\alt(m)}(g_n)) = \prod_{g_n \neq 1} \mu(\mathcal{F}_{\alt(m)}(g_n)) \leq \prod_{g_n \neq 1} c.$$
It follows that $\mu(\mathcal{F}_G(g)) > 0$ if and only if $g_n$ is almost always $1$, so $\mathcal{F}(G)$ equals the restricted direct power $\alt(m)^{(\mathbb{N})}$, and hence it is not closed in $G$ (it is dense).

Now assume that $\mathcal F$ does not contain all the finite cyclic groups. By our assumption on $\mathcal F$, there exist two prime numbers $p,q$ such that $C_p$ is in $\mathcal F$ and $C_q$ is not in $\mathcal F$. Let $n$ be the order of $q$ in the multiplicative group of $\mathbb{Z}/p\mathbb{Z}$, so that $p$ divides $q^n-1$. Let $Q$ be the additive group of the field with $q^n$ elements. The group $P \cong C_p$ generated by an element of multiplicative order $p$ in the field with $q^n$ elements acts on $Q$ fixed point freely (by multiplication), so $H=Q \rtimes P$ is a Frobenius group. Fix a positive integer $t$ and let $n_t=p^t$, $\sigma=(1,\ldots,n_t) \in S_{n_t}$, $h \in H \setminus Q$, $Y_t= Q^{n_t} \langle g \rangle < H \wr \langle \sigma \rangle$ where $g=(h,1,\ldots,1)\sigma$. Note that $|Y_t| = q^{p^t n} p^{t+1}$ and that a subgroup of $Y_t$ belongs to $\mathcal{F}$ if and only if it is a $p$-subgroup.
We want to prove that, if $u \in Y_t$, then $\mu(\mathcal F_{Y_t}(u)) \leq q^{-{n_tn}}.$ %being $\mu(\mathcal{F}_{Y_t}(u))$
%the probability that a randomly chosen $w \in Y_t$ is such that $\langle u, w\rangle \in \mathcal{F}$.
Notice that $\langle g\rangle $ is a $p$-Sylow subgroup of $Y_t$ and that $g^{n_t}=(h,\dots,h)$ generates the unique subgroup of order $p$ of $\langle g \rangle.$ Moreover
$(y_1,\dots,y_{n_t})\in Q^{n_t}$ normalizes $\langle g^{n_t}\rangle $ if and only if $y_i=1$ for every $1\leq i\leq n_t$ and therefore $\langle g\rangle$ is the unique $p$-Sylow subgroup of $Y_t$ containing $g^{n_t}.$ This implies that  $\langle g^{n_t}, x\rangle \in \mathcal{F}$ if and only if $x \in \langle g\rangle$ and therefore $$\mu(\mathcal{F}_{Y_t}(g^{n_t}))=\frac{|g|}{|Y_t|}=\frac{1}{q^{n_tn}}.$$
Let $u$ be an arbitrary nontrivial $p$-element of $Y_t$. 
Since $\mu(\mathcal{F}_{Y_t}(u))\leq \mu(\mathcal{F}_{Y_t}(u^m))$ for every $m\in \mathbb N,$ we may assume $|u|=p$. But then $u$ is conjugate to an element of $\langle g^{n_t} \rangle$, and therefore
$$\mu(\mathcal F_{Y_t}(u))\leq \frac{|g|}{|Y_t|}=\frac{1}{q^{n_tn}}.$$
Now consider $G=\prod_{t\in \mathbb N}Y_t$.
Assume $y=(y_t)_{t\in \mathbb N}\in G$ and let $\Lambda:=\{t\mid y_t\neq 1\}.$ Then
$$\mu(\mathcal F_{G}(y))\leq \prod_{t\in \Lambda}\frac{1}{q^{n_tn}},$$ so if $\mu(\mathcal{F}_{G}(y))>0$, then $|\Lambda|$ is finite (and clearly $y_t$ is a $p$-element for every $t\in \Lambda$). Conversely, suppose that there exists a finite $\Lambda \subset \mathbb{N}$ such that $y_t=1$ if $t\notin \Lambda,$ $y_t$ is a $p$-element of $Y_t$ otherwise. If $t\in \Lambda,$ then $\alpha_t:=\mu(\mathcal{F}_{Y_t}(y_t))>0$. On the other hand, since every element of $Y_t \setminus Q^{n_t}$ is a $p$-element, $$\mu(\mathcal{F}_{Y_t}(1)) \geq \frac{|Y_t|-|Q^{n_t}|}{|Y_t|} = 1-\frac{1}{p^{t+1}}.$$ Hence
$$\mu(\mathcal F_{G}(y))\geq \prod_{t\in \Lambda}\alpha_t
\prod_{t\notin \Lambda}\left(1-\frac{1}{p^{t+1}}\right)>0.$$
We have proved that  $\mathcal{F}(G)$ equals the {set of $p$-elements belonging to the restricted direct product of the groups $Y_t$}. Any $p$-element in $G$ belongs to the closure of $\mathcal{F}(G)$, so it is not closed. 
\end{proof}

On the other hand, we are able to find a family $\mathcal F$ for which $\mathcal F(G)$ is closed for \emph{every} profinite group $G$. {If $p$ is a prime number, let $\Omega_p(G)$ be the set of $p$-elements of $G$.}

\begin{prop} \label{exp2}
Let $\mathcal F$ be the class of all finite groups of exponent at most $2$. Then $\mathcal{F}(G)$ is closed in $G$ for every profinite group $G$.
\end{prop}
\begin{proof}
If $x,y \in G$ are such that $y \in \mathcal{F}_G(x)$ then $x,y$ have order at most $2$ and $xy=yx$. In particular, if $x \in \mathcal{F}(G)$, i.e. $\mu(\mathcal{F}_G(x))>0$, then the probability that an element of $G$ has order $2$ is positive. So we may assume that $\mu(\Omega_2(G))>0$, since otherwise $\mathcal{F}(G) = \varnothing$ is closed. By \cite[Theorem 2]{LP}, $G$ is abelian-by-finite, i.e. there exists an abelian open normal subgroup $N$ of $G$. Write $G=t_1N \cup \ldots \cup t_rN$ with $r=|G:N|$. If $t_iN$ contains elements of order $2$, then we may assume $t_i$ has order $2$ and the set of elements of order $2$ in $t_iN$ coincides with $t_iH_i$ where $H_i=\{n \in N\ |\ n^{t_i}=n^{-1}\} \leq N$. We assume that $t_i$ has order $2$ precisely when $1 \leq i \leq s$, where $s \leq r$. We want to prove that $\mathcal{F}(G)$ is closed. Since $\mathcal{F}(G)$ consists of elements of order at most $2$, it is enough to prove that $\mathcal{F}(G) \cap t_jH_j$ is closed for all $j=1,\ldots,s$. Without loss of generality, it is enough to prove that $Y=\mathcal{F}(G) \cap t_1H_1$ is closed. Let $x=t_1h_1 \in Y$ where $h_1 \in H_1$. There must exist $j$ such that $\mu(\mathcal{F}_G(x) \cap t_jH_j) > 0$. Let $$\Lambda_j := \{x \in Y\ |\ \mu(\mathcal{F}_G(x) \cap t_jH_j) > 0\}.$$
Then $Y=\Lambda_1 \cup \ldots \cup \Lambda_s$, so it is enough to prove that $\Lambda_j$ is closed for each $j=1,\ldots,s$. Assume $x = t_1h_1 \in \Lambda_j$. Note that the fact that $\Lambda_j \neq \varnothing$ implies that $H_j$ is a closed subgroup of $G$ with positive Haar measure, so $|G:H_j|$ is finite. If $t_jh_j \in t_jH_j$, then saying $t_jh_j \in \mathcal{F}_G(x)$ is equivalent to writing $t_1h_1t_jh_j=t_jh_jt_1h_1$. Since $H_1$ and $H_j$ are contained in $N$, which is abelian, we can restate this condition as $t_1 t_j^{h_1^{-1}} = t_j t_1^{h_j^{-1}}$. Moreover, $H_j \cap C_G(t_1h_1) = C_{H_j}(t_1)$ so $t_j^{-1} (\mathcal{F}_G(x) \cap t_jH_j)$ is a coset of $C_{H_j}(t_1)$ in $H_j$. So if $|G:C_{H_j}(t_1)|$ were infinite, then $\mu(\mathcal{F}_G(x) \cap t_jH_j)=0$ would imply $x \not \in \Lambda_j$, a contradiction. Therefore, if $\Lambda_j \neq \varnothing$, then $|G:C_{H_j}(t_1)|$ is finite. Thus, on one hand, we can conclude that all elements $t_1h_1 \in \Lambda_j$ are contained in $X_j =\{t_1h_1 : h_1 \in H_1 \text{ and }t_1t_j^{h_1^{-1}} \in t_jt_1^{H_j}\}$. On the other, if $x = t_1h_1 \in X_j$, from $\mu(\mathcal{F}_G(x) \cap t_jH_j) = \mu(C_{H_j}(t_1)) = 1/|G:C_{H_j}(t_1)| > 0$ it follows that $x$ is in $\Lambda_j$. Therefore, $\Lambda_j$ equals one of $\varnothing$, $X_j$, hence it is closed.
\end{proof}

Notice that the above proof relies heavily on the fact that, if $G$ satisfies the identity $x^2=1$ with positive probability, then it is virtually abelian. It is natural to wonder what happens for a class $\mathcal F$ of finite groups with uniformly bounded exponent, but we can't use the same strategy as above since we lack a similar structural result for profinite groups which satisfy $x^n=1$ with positive probability for $n>2$.

\begin{question}
    Let $\mathcal F$ be a class of finite groups closed under taking subgroups, quotients and direct products and assume that it has uniformly bounded exponent. Is there a profinite group $G$ such that $\mathcal F(G)$ is not closed?
\end{question}

\section{The class of finite solvable groups}\label{solv}

In this section, we discuss the case in which $\mathcal F$
coincides with the class $\mathcal S$ of the finite solvable groups. The set $\mathcal S_G(x)$ consists of all the  elements $y\in G$ such that the subgroup generated by $x$ and $y$ is prosolvable and is called the \emph{solvabilizer} of $x$ in $G.$ The characterization of elements with solvabilizer  of positive measure was proposed in \cite[Conjecture 1]{solv} and proved in \cite[Theorem 6]{solv} and \cite[Corollary 1.3]{fgg}. {The authors of these papers} obtained the following: given a profinite group $G$ and a positive integer $n$, let $\Sigma_n(G)$ be the set of the elements $g$ in $G$  that centralize all but at most $n$ nonabelian factors in a chief series of $G/N$ for every open normal subgroup $N$ of $G.$ Then, we have that $\mathcal S(G)=\cup_{n\in \mathbb N}(\Sigma_n(G))$, answering \cref{q1} and \cref{q3}: $\mathcal{S}(G)$ is a subgroup of $G$ which admits a group-theoretical characterization. For \cref{q2}, it suffices to use \cref{allfingrps} to deduce that $\mathcal{S}(G)$ is not closed in $G$ in general.

Our first aim in this section is to use the aforementioned characterization and obtain a better understanding of the structure of $\mathcal S(G).$ If $G$ is a profinite group, then we will denote by $R(G)$ the largest subgroup of $G$ which is closed, normal and prosolvable. It is well known that $g\in R(G)$ if and only if $gN\in R(G/N)$ for every open normal subgroup $N$ of $G.$

\begin{lemma}
Let $G$ be a finite group. If $g$ centralizes all the nonabelian chief factors of $G$, then $g\in R(G).$
\end{lemma}
\begin{proof}
	We prove the statement by induction on $|G|.$ Let $N$ be a minimal normal subgroup of $G$ and let $K/N=R(G/N)$. By induction $g\in K.$ If $N$ is abelian, then $K=R(G)$ and we are done. If $N$ is nonabelian, then $C_K(N)\cap N=Z(N)=1.$ Thus $C_K(N)$ is solvable, and therefore $C_K(N)\leq R(K)\leq R(G).$ Since $g\in C_K(N),$ we are done.
\end{proof}

From now on, if $G$ is a profinite group and $g \in G$, we will write that $g$ centralizes all the nonabelian chief factors of $G$ to mean that $gM$ centralizes all the nonabelian chief factors of $G/M$ for every open normal subgroup $M$ of $G.$

\begin{cor}\label{corradical}
Let $G$ be a profinite group. If $g \in G$ centralizes all the nonabelian chief factors of $G$, then $g\in R(G).$
\end{cor}

\begin{lemma}\label{fccenter}Let $G$ be a profinite group. Then $R(G) \leq \mathcal S(G)$. Moreover $\mathcal S(G)/R(G)$ is contained in the FC-center of $G/R(G)$ and all elements of $\mathcal S(G)/R(G)$ have finite order.
\end{lemma}
\begin{proof}
For every $g \in R(G)$ and every $x \in G$, we have that $\la x,g\ra \leq R(G)\langle x \rangle$, which is prosolvable, so $\mathcal S_G(g) = G$ and $R(G) \leq \mathcal S(G)$. Thus is not restrictive to assume $R(G)=1$. Let $g\in \mathcal S(G).$ There exists an open normal subgroup $N$ of $G$ such that $g$ centralizes all the nonabelian chief factors of $N.$ Let $K=N\langle g\rangle$. Then, by \cref{corradical}, $g\in R(K).$ Since $R(G)=1$, we must have $R(N)=1$. In particular this implies $R(K)\cap N=1$, so $[R(K),N] = 1$. It follows that $N\leq C_G(g)$, therefore $g$ is contained in the $FC$-center of $G$ . Moreover $g$ has finite order since $\langle g\rangle \cap N = 1$.
\end{proof}

\begin{lemma}\label{inclus}Let $G$ be a profinite group. Then the $FC$-center of $G$ is contained in $\mathcal S(G).$
\end{lemma}

\begin{proof}
Let $g$ be an element of the $FC$-center of $G$. Then $C_G(g)$ is open and therefore it contains an open normal subgroup $N$ of $G$. Clearly $g$ centralizes all the nonabelian chief factors of $G$ contained in $N$ and therefore $g\in \mathcal S(G).$
\end{proof}

Recall that if $G$ is a profinite group, $N \unlhd_c G$ and $\pi \colon G \to G/N$ is the natural projection, then $\mu_{G/N}(S) = \mu_G (\pi^{-1}(S))$ for every measurable subset $S$ of $G/N$. If $R=R(G)$, then one can prove that $\pi^{-1}(S_{G/N}(gR)) = S_G(g)$ and so $\mathcal S(G/R(G)) = \mathcal S(G)/R(G)$.

\begin{proof}[Proof of \cref{descs}]
The first part of the statement  follows from Lemmas \ref{fccenter} and \ref{inclus}. Furthermore, if $\mathcal S(G)$ is closed, then $\mathcal S(G)/R(G)$ is a profinite FC-group and contains an abelian open normal subgroup $N/R(G)$ by \cite[Lemma 2.6]{sh}. The preimage of $N$ is open and prosolvable in $\mathcal S(G)$, so $\mathcal S(G)/R(G)$ is finite.
\end{proof}

Note that when $R(G)=1$ every element in $\mathcal S(G)$ has finite order. This immediately gives examples of profinite groups $G$ such that $\mathcal S(G) = 1$; it suffices to consider torsion-free profinite groups with trivial prosolvable radical (\emph{e.g.}, non-abelian free profinite groups). For examples of groups $G$  with $\mathcal S(G)=1$ but containing non-trivial elements of finite order, one may look at the hereditarily just-infinite, non-virtually prosolvable groups obtained by infinitely iterated wreath products (see \cite{mv}).

\

 Before proving \cref{musolpos}, we need a preliminary lemma.

% \begin{thm}\label{sigmapos}
% 	Let $G$ be a profinite group. If $\mu(\mathcal S(G))>0,$ then $G$ is virtually prosolvable, and consequently $\mathcal S(G)=G.$
% \end{thm}
% Before proving this theorem, we need a preliminary lemma.

\begin{lemma}\label{subseries}
	Let $G$ be a finite group, $H$ a normal subgroup of $G$ and $N$ a nonabelian minimal normal subgroup of $G$ contained in $H$. If $M$ is a normal subgroup of $H,$ then an $H$-chief series of $NM/M$ has length at most $|G:H|.$
\end{lemma}

\begin{proof}
	We have that $N=S_1\times \dots \times S_u,$ where $S_1,S_2,\dots,S_u$ are isomorphic nonabelian simple groups.
	Since $N$ is a minimal normal subgroup of $G$, $G$ permutes transitively the factors $\{S_1,\dots,S_u\}.$ Since $H$ is normal in $G,$ the orbits of $H$ on $\{S_1,\dots,S_u\}$ have all the same size, and the number of these orbits is at most $t=|G:H|.$ Notice that $NM/M \cong_H N/N\cap M$ and a minimal $H$-subgroup of $N/N\cap M$ corresponds to an $H$-orbit on $\{S_1,\dots,S_u\}.$ Thus an $H$-chief series of $NM/M$ has length at most $t.$
\end{proof}
 
\begin{proof}[Proof of \cref{musolpos}]
	Assume $\mu(\mathcal S(G))>0.$ Since $\mathcal S(G)$ is a subgroup of $G$, we must have that $\mathcal S(G)$ has finite index in $G$. Moreover by \cite[Lemma 1.1 (2)]{bm} $\mathcal S(G)$ is closed in $G.$ Let $t=|G:\mathcal S(G)|$  and let $g\in \mathcal S(G)$. Then there exists $m\in \mathbb N$, such that, for every open normal subgroup $N$ of $G,$ $g$ centralizes all but $m$ nonabelian factors in any chief series of $G/N.$ Now let $M$ be an open normal subgroup of $\mathcal S(G)$. The normal core $K$ of $M$ in $G$ is open in $G$. We may so find a chief series
	$1=N_0/K < \dots < N_u/K=\mathcal S(G)/K<\dots  < N_v/K=G/K$ of $G/K$. Then
	$N_0M/M \leq \dots \leq N_uM/M=\mathcal S(G)/M$ is a normal chain of $\mathcal S(G)/M.$ Notice that $N_iM/N_{i+1}M$ is not necessarily a chief factor of $\mathcal S(G)/M$, but by Lemma \ref{subseries} if $N_i/N_{i+1}$ is nonabelian, then the length of a $\mathcal S(G)$-chief series of
	$NM/M$ is at most $t$. Moreover if $g$ centralizes $N_iM/N_{i+1}M$, then it centralizes all the factors of a $\mathcal S(G)$-chief series of $N_iM/N_{i+1}M.$ Since $g$ centralizes all but $m$ of the nonabelian factors $N_i/N_{i+1}$, we conclude that $g$ centralizes all but $m\cdot t$  nonabelian factors of a chief series of $\mathcal S(G)/M.$ We have so proved that $g\in \mathcal S(\mathcal S(G)),$ that is, $\mathcal S(\mathcal S(G))=\mathcal S(G).$ By \cite[Theorem 1.2]{aemp}, $\mathcal S(G)$ is virtually prosolvable. But then, since $\mathcal S(G)$ is open in $G,$ $G$ also is virtually prosolvable.
\end{proof}

%Another generalization of \cite[Theorem 1.2]{aemp} is the following, done by considering solely the solvabilizer of elements of odd order.

We need a preliminary lemma for the proof of \cref{thodd}.

\begin{lemma}\label{oddele}
Let $N$ be a nonabelian minimal normal subgroup of a finite group $G$. If $g\in G$ and the order of $g$ modulo $N$ is odd, then there exists $n\in N$ such that $|gn|$ is odd and $gn \notin C_G(N).$
\end{lemma}

\begin{proof}
We may write $g=g_1g_2$ with $|g_1|$ odd, $|g_2|$ a 2-power and $[g_1,g_2]=1.$ Since $g$ has odd order modulo $N,$ $g_2\in N.$ So $g_1\in gN.$ This means that we may assume that $|g|$ is odd. If $g\notin C_G(N)$ we are done. Assume that $g\in C_G(N)$. In this case let $n$ be a non-trivial element of odd order in $N$. Then $|gn|$ is odd, and $gn\notin C_G(N)$ (since $n\notin C_G(N))$.
\end{proof}

%\begin{proof}[Proof of \cref{thodd}]
%With iterated applications of Lemma \ref{oddele}, arguing as in the proof of \cite[Corollary 22]{pel}, we may construct an element $g\in G$ of odd order and which centralizes no nonabelian chief factor of $G.$ By assumption, the solvabilizer of $g$ has positive Haar measure and therefore $g$ centralizes almost all the nonabelian chief factors of $G$. It follows that a chief series of $G$ \textcolor{blue}{(Martino) Abbiamo definito cos'\`e una chief series di un gruppo profinito?} \textcolor{green}{(Nowras) No, e penso che non sia possibile definirlo bene per gruppi profiniti qualsiasi. La riscrivo con più parole per essere più chiaro.} contains only finitely many nonabelian factors and therefore $G$ is virtually prosolvable.
%\end{proof}

\begin{proof}[Proof of \cref{thodd}]
We prove the contrapositive. Arguing as in the proof of \cite[Corollary 22]{pel} and using \cref{oddele}, we construct an element $g\in G$ of odd order and which centralizes no nonabelian chief factor of $G$. We repeat the argument as follows. If $G$ were not virtually prosolvable, we would be able to find a descending chain of open normal subgroups
\[G \geq N_1 > M_1 \geq N_2 > M_2 \geq \ldots\]
such that $N_i/M_i$ is a nonabelian minimal normal subgroup of $G/M_i$. It is not restrictive to suppose that $\cap_{i\in \mathbb N} M_i = 1$, since for any $g \in G$ and $K\unlhd_c G$, we have $\mu(\mathcal S_G(g)) \leq \mu(\mathcal S_{G/K}(gK))$. Thus, we may assume $G \cong \varprojlim_{i \in \mathbb N} G/M_i$ and the elements of $G$ will be of the form $(g_iM_i)_{i \in \mathbb N}$ such that $g_iM_i = g_jM_i$ whenever $j\geq i$. Take an element $g_1 \in G$ of odd order and consider $g_1M_2$. Its order modulo $N_2/M_2$ is odd, so by \cref{oddele} we may can choose $g_2 \in g_1N_2$ such that $g_2M_2$ does not centralize $N_2/M_2$. It is clear that $g_2 M_1 = g_1M_1$. By induction, we obtain elements $g_iM_i$ such that $g=(g_iM_i)$ is an element of $G$. Further, for every positive integer $n$, we can find $M\unlhd_o G$ such that 
$g$ does not centralize at least $n$ nonabelian chief factors of $G/M$. It follows that $\mu(\mathcal S_G(g))=0$.
\end{proof}

Finally, we make a remark regarding a result of Larsen and Shalev. In \cite{ls}, they prove the following: {\sl{Let $G$ be a profinite group and suppose that the set of elements of $G$ of finite odd order has positive Haar measure. Then $G$ has a prosolvable open normal subgroup.}} This motivates the following question:	{\sl{Let $G$ be a profinite group and suppose that the set of elements whose order, as a supernatural number, is odd has positive Haar measure. Is $G$ virtually prosolvable?}} We give a counterexample as follows.

\

Let $G=\prod_{n \geq 2}{\rm{SL}}(2,2^n)$ and denote by $\mathcal S_{\rm{odd}}$ the class of odd order finite solvable groups. Since $\mathcal S_{\rm{odd}}(G) \subseteq \mathcal S (G)$
and $R(G)=1,$ it follows from \cref{descs} that 
$\mathcal S_{\rm{odd}}(G) \subseteq \mathcal S(G)$ is contained in the $FC$-center of $G$, which coincides with the restricted direct product of the factors.
Moreover if $g=(g_n)_{n\in \mathbb N}\in \mathcal S_{\rm{odd}}(G)$, then $|g_n|$ is odd for every $n\in \mathbb N$ (otherwise $\mathcal S_{{\rm{odd}},_G}(g) =\varnothing).$ Thus 
$\mathcal{S}_{\rm{odd}}(G)$ is contained in the subset
$\Omega_{\rm{odd}}(G)$ consisting of the elements $(g_n)_{n\in \mathbb N}$
such that $|g_n|$ is odd for every $n\in \mathbb N$
and $g_n=1$ for all but finitely many $n\in \mathbb N.$ We claim that the equality $\mathcal{S}_{\rm{odd}}(G) = \Omega_{\rm{odd}}(G)$ holds. The group ${\rm{SL}}(2,2^n)$ has order $2^n(2^{2n}-1)$ and it contains precisely $2^{2n}-1$ nontrivial $2$-elements, %(see for example \cite{moller}) 
all of them have order $2$. Moreover, the order of an arbitrary element of ${\rm{SL}}(2,2^n)$ is either $2$ or odd. So the probability that an element of ${\rm{SL}}(2,2^n)$ has odd order is precisely $1-\frac{1}{2^n}.$
Then the probability that a randomly chosen element of $G$ has odd order (as a supernatural number) is
$$\eta:=\prod_{n \in \mathbb N}\left(1-\frac{1}{2^n}\right)>0.$$
It follows that the probability of randomly choosing an odd order element of $G$ is positive while $G$ is not virtually prosolvable.

\section{The class of finite nilpotent groups}\label{nilpot}

In this section we prove Theorems \ref{nilpp}, \ref{localnilp}, \ref{classp} and \ref{nofiniteorder}. If $G$ is a profinite group and $x\in G,$ we will denote by $\nil_G(x)$ the \emph{nilpotentizer} of $x$ in $G$, the set of elements $y \in G$ such that $\langle x,y \rangle$ is pronilpotent. Given a finite group $G$, a normal subgroup $N$ of $G$ and an element $x \in G$, the first thing we do is find a way to nicely bound $\mu_G(\mathcal N_G(x))$ from above in terms of $\mu_{G/N}(\mathcal N_{G/N}(xN))$.

\begin{lemma}\label{nilabcf}
	Let $G$ be a finite group, let $p$ be a prime and suppose that $N$ is an abelian minimal normal subgroup of $G$ whose order is a $p$-power. Let $x\in G$ and write
	$x=x_1x_2$ where $|x_1|$ is a $p$-power, $|x_2|$ is coprime with $p$ and $[x_1,x_2]=1.$
	Then
	$$\frac{|\nil_G(x)|}{|G|}\leq \frac{|\nil_{G/N}(xN)|}{|G/N|}\frac {|C_N(x_2)|}{|N|}.$$
\end{lemma}

\begin{proof}
   Given $y\in \nil_G(x)$ we want to estimate the size of $\Delta:=\nil_G(x)\cap Ny.$ We write $x=x_1x_2$ and $y=y_1y_2$, where $|x_1|$, $|y_1|$ are $p$-powers, $|x_2|, |y_2|$ are coprime with $p$ and $[x_1,x_2]=[y_1,y_2]=1.$ Since $\langle x,y
	\rangle$ is nilpotent, $[x_2,y_1]=[x_1,y_2]=1.$  Now assume $z\in \Delta.$ Then there exist $n_1, n_2\in N$ such that $z=y_1n_1y_2n_2,$ where $|y_1n_1|$ is a $p$-power, $|y_2n_2|$ is coprime with $p$ and $[y_1n_1,y_2n_2]=1.$
	First notice that $H = \langle x_2,y_2 \rangle$ and $K = \langle x_2, y_2n_2\rangle$ are Hall $p'$-subgroups of the solvable group $\langle x_2,y_2 \rangle N$, so by Hall's theorem there exists $n \in N$ such that $H^n = K$. Since $N \cap K = \{1\}$, if $h \in H$ then $|hN \cap K| \leq 1$, moreover $h^n \in hN$, so ${x_2}^n = x_2$ and ${y_2}^n = y_2n_2$. In other words, $y_2n_2 = {y_2}^{n}$ for some
	$n\in C_N(x_2).$ Thus the number of possible choices for
	$y_2n_2$ is at most $|C_N(x_2)|/|C_N(x_2)\cap C_N(y_2)|.$ Moreover
	$\langle x_2, y_1n_1\rangle$ nilpotent implies that $n_1\in C_N(x_2).$ Finally, using the commutator equalities $[ab,c]=[a,c]^b [b,c]$ and $[c,ab]=[c,b][c,a]^b$, since $N$ is abelian we have 
    \begin{align*}
        1 & = [y_1n_1,y_2n_2] = [y_1,n_2]^{n_1} [y_1,y_2]^{n_2n_1} [n_1,n_2] [n_1,y_2]^{n_2} = [y_1,n_2][n_1,y_2],
    \end{align*}
    thus $[y_1,n_2]=[y_2,n_1].$ Once $n_2$ is fixed, we count how many $n_1$ there are in $C_N(x_2)$ with $[y_1,n_2]=[y_2,n_1].$ It is possible that there is no $n_1$ satisfying this equation. In any case assume that 
	$[y_2,n_1]=[y_2,\tilde n_1]$, where $n_1$ and $\tilde n_1$ are two different elements of $C_N(x_2).$ Then $n_1\tilde n_1^{-1}\in C_N(y_2)\cap C_N(x_2)$, so, given $n_2$ there are at most $|C_N(x_2)\cap C_N(y_2)|$ possible choices for $n_1$. We conclude
	that the number of possibilities for the pair $(n_1,n_2)$
	is at most
	$$\frac{|C_N(x_2)|}{|C_N(x_2)\cap C_N(y_2)|}\cdot |C_N(x_2)\cap C_N(y_2)|=|C_N(x_2)|.
	$$
	Hence $|\Delta|\leq |C_N(x_2)|$. This implies the result.
\end{proof}

In the next proofs we will use the following notation. Given a profinite group $G$ and an element $g \in G$, for every prime $p$, $g=g_pg_{p^\prime}$ with $g_p$ a $p$-element, $g_{p^\prime}$ a $p^\prime$-element and $[g_p,g_{p^\prime}]=1.$ In particular, $g_p$ and $g_{p'}$ belong to $\langle g \rangle$. For every $g\in G$ and every open normal subgroup $N$ of $G$ we define the following non-negative integers:
\begin{enumerate}
	\item $\tau_{nonab}(g,N)$ is the number of nonabelian factors $X/Y$ in a chief series of $G/N$ such that $[X,g]\not\leq Y.$
	\item $\tau_{ab,p}(g,N)$ is the number of abelian factors $X/Y$ in a chief series of $G/N$ of order a $p$-power for some prime $p$ and such that $[X,g_{p^\prime}]\not\leq Y.$
	\item $\tau(g,N)=\tau_{nonab}(g,N)+\sum_p\tau_{ab,p}(g,N)$.
\end{enumerate}
	
	\begin{cor}\label{cencf}
		Let $G$ be a profinite group and let $g \in G$. If $\mu(\nil_G(g))>0,$ then there exists $u\in \mathbb N$ such that $\tau(g,N)\leq u$ for every open normal subgroup $N$ of $G.$
	\end{cor}
	
\begin{proof}
If $\mu(\nil_G(g))>0,$ then $\mu(\mathcal S_G(g))>0$ so by \cite[Corollary 1.3]{fgg} $g$ centralizes almost all the nonabelian chief factors of $G$, in other words there exists a constant $C$ such that $\tau_{nonab}(g,N) \leq C$ for every $N \unlhd_{\circ} G$. Moreover it follows from Lemma \ref{nilabcf} that, for every prime $p$, $g_{p^\prime}$ centralizes almost all the abelian chief factors of $G$ whose order is a power of $p$. Explicitly, if $N \unlhd_{\circ} G$ and $N = N_k \unlhd \cdots \unlhd  N_1 = G$ is a chief series of $G/N$, then, whenever $N_i/N_{i+1}$ is abelian, set $c_i := |C_{N_i/N_{i+1}}(g_{p^\prime})|/|N_i/N_{i+1}|$, being $p$ the unique prime divisor of $|N_i/N_{i+1}|.$
Note that, if $c_i < 1$, then $c_i \leq 1/2$, so $0 < \varepsilon = \mu(\nil_G(g)) \leq (1/2)^t$ where $t=\sum_p\tau_{ab,p}(g,N).$ It follows that $$\tau(g,N) = \tau_{nonab}(g,N) + t \leq C+\log_2(1/\varepsilon).\qedhere$$
\end{proof}

We need the following revised version of \cite[Theorem 25]{pel}.

\begin{thm}\label{thm_dh} Let $G$ be a finite group and $p$ a prime divisor of $G$. Then a $p$-element $g$ of $G$ centralizes all the abelian $p'$-chief factors and all the nonabelian chief factors of $G$ if and only if $g \in O_p(G).$
\end{thm}
\begin{proof}
Since the Fitting subgroup of $G$ is generated by the subgroups $O_p(G)$ where $p$ is a prime dividing $|G|$, the implication $(\Leftarrow)$ follows from \cite[Theorem 13.8(b)]{dh}. Suppose that a $p$-element $g$ of $G$ centralizes all the abelian $p'$-chief factors and all the nonabelian chief factors of $G$. This is equivalent to saying that $g$ centralizes all the chief factors of $G$ whose order is not a $p$-power.
	Let $q$ be a prime divisor of $|G|$ and let $C_q(G)$ be the intersection of the centralizers of the chief factors of $G$ whose order is divisible by $q
	.$ By \cite[Theorem 13.8(a)]{dh}, we have that $C_q(G)=O_{q^\prime,q}(G)$, where $O_q(G/O_{q^\prime}(G)) = O_{q^\prime,q}(G)/O_{q^\prime}(G)$. So it is sufficient to prove that a $p$-Sylow subgroup $P$ of $C=\cap_{q\neq p}O_{q^\prime,q}(G)$ is normal in $C$ and coincides with $O_p(G).$
	Since $O_{q^\prime,q}(G)$ is $q$-nilpotent, it follows from \cite[Theorem 13.4(a)]{dh}
	that $C$ is $q$-nilpotent for every $q\neq p.$ If $K_q$ is the normal $q$-complement in $C$, then $P=\cap_{q\neq p}K_q$.  Since $C$ is a normal subgroup of $G$, it follows that $P\leq O_p(G).$ On the other hand $O_p(G)\leq O_{q^\prime,q}(G)$ for every $q\neq p,$ so $O_p(G)\leq P.$
\end{proof}

So arguing as in \cite[Theorem 27]{pel}, we obtain the following.

% \begin{thm}
% 	Let $G$ be a profinite group and let $g$ be an element of $G$. %Suppose either that $p$ is odd or that $G$ is virtually prosolvable. 
% 	If $\mu(\nil_G(g))>0,$ then there is an open subgroup $H$ of $G$ such that  $g \in \fit(H).$ In particular, $g$ has finite order modulo $\fit(G).$ 
% \end{thm}
\begin{proof}[Proof of \cref{nilpp}]
It follows from \cref{cencf} that $G$ contains an open normal subgroup $M$ such that, for every prime $p$, $g_p$  centralizes all the chief factors of $M$ that are either nonabelian or have order coprime to $p.$ Hence, by \cref{thm_dh}, $g_p\in O_p(M\langle g\rangle)$. This is because, if $H$ is a closed subgroup of $G$, $P$ varies in the family of $p$-Sylow subgroups of $H$ and $N$ varies in the family of the open normal subgroups of $H$, then $O_p(H) = \cap_P P = \cap_P \cap_N PN = \cap_N \cap_P PN$ and $\cap_P PN$ is the preimage of $O_p(HN/N)$ in $H$. This implies that $g\in \fit(M\langle g\rangle).$
Moreover, $\langle g\rangle \cap M\leq \fit(M \langle g \rangle) \cap M = \fit(M)\leq \fit(G)$ so the order of $g$ modulo $\fit(G)$ is at most $|G/M|.$    
\end{proof}

In the proof of the next lemma we will use the following definition: if $K$ is a closed normal subgroup of $G$ and $p$ is a prime, we will denote by $\Delta_{G,p}(K)$ the set of the elements $g\in G$ with the following property: if $Y$ is an open normal subgroup of $K$, $X/Y$ is a chief factor of $K/Y$
and either $X/Y$ is nonabelian or has order coprime with $p$, then $[g,X]\leq Y.$

Recall that a profinite group is countably based if and only if it admits a countable family of open normal subgroups with trivial intersection.

\begin{lemma}\label{nilpel}
	Let $G$ be a countably based profinite group and let $P$ be a $p$-Sylow subgroup of $G.$  Suppose that $\mu(\nil_G(g))>0$ for every $g\in P.$
    Then the following hold:
    \begin{enumerate}
    \item $N_G(P)$ is an open subgroup of $G.$
    \item $P/O_p(G)$ is finite.
    \end{enumerate}
\end{lemma}
\begin{proof}
Since $G$ is countably based, there exists a descending chain  $\{N_i\}_{i\in \mathbb N}$ of open normal subgroups of $G$ such that $N_0=G,$  $\cap_{i\in \mathbb N}N_i=1$ and $N_i/N_{i+1}$ is a chief factor of $G/N_{i+1}$ for every $i\in \mathbb N.$    
Let $\Delta$ be the set of the natural numbers $u \in \mathbb{N}$ such that $N_{u}/N_{u+1}$ is a $p'$-chief factor or a nonabelian chief factor and let $\Lambda$ be the set of cofinite subsets of $\Delta$. For every $\lambda \in \Lambda$, we have that $C_\lambda := \cap_{i \in \lambda} C_P(N_i/N_{i+1})$ is a closed subgroup of $P$ and, by \cref{cencf}, it follows that $P= \cup_\lambda C_\lambda$. By the Baire Category Theorem, there exist $\lambda \in \Lambda$, an open normal subgroup $M$ of $P$ and $g \in P$ such that $gM \subseteq C_\lambda$. Since $C_\lambda$ is a subgroup of $G,$ this implies that $M$ itself is contained in $C_\lambda$. In particular, we can find an open normal subgroup $N$ of $G$ such that $M \leq \Delta_{G,p}(N)$ and therefore $M\cap N\leq O_p(N)\leq O_p(G).$ Hence $MO_p(G)/O_p(G)\cong M/(M\cap O_p(G))$ is an epimorphic image of 
$M/(M \cap N) \cong MN/N$ and therefore is finite, so $P/O_p(G)$ is finite too, which proves item (2). Write $P/O_p(G)=\{g_1O_p(G), \ldots, g_tO_p(G)\}$. By  \cref{cencf}, for each $1\leq i\leq t$, there exists an open normal subgroup $N_i$ of $G$ such that $g_i \in \Delta_{G,p}(N_i)$. Let $L = N \cap N_1 \cap \dots \cap N_t.$ It follows from \cref{thm_dh} that $P \leq O_p(PL)$ and, since $P$ is a Sylow subgroup of $G$, equality holds. Thus, $N_G(P)$ contains $L$ and consequently $N_G(P)$ is open in $G.$ This proves item (1).
\end{proof}

\begin{lemma}\label{intersecochiusi}Let $H$ and $K$ be closed subgroups of a profinite group $G$. If $H\cap K$ has infinite index in $H$ then there exists a closed normal subgroup $N$ of $G$ such that
\begin{enumerate}
\item $G/N$ is countably based;
\item $HN \cap KN$ has infinite index in $HN.$
\end{enumerate}
\end{lemma}
\begin{proof}
Let $\mathcal N$ be the family of the open normal subgroups of $G$. We have $H=\cap_{R\in \mathcal N}HR$ and $K=\cap_{S\in \mathcal N}KS,$
hence $$H \cap K=\cap_{R,S \in \mathcal N}HR\cap KS=\cap_{R,S\in \mathcal N}H(R\cap S) \cap K(R\cap S)=\cap_{T\in \mathcal N}HT\cap KT.$$
Since the index of $H\cap K$ in $H$ is infinite, there exists a countable subset $\{x_n\}_{\in \mathbb N}$ of $H$ such that $x_i(H\cap K) \neq x_j(H\cap K)$ whenever $i\neq j.$ In particular
for every $i\neq j,$ there exists an open normal subgroup $N_{i,j}$ of $G$ such that $x_i^{-1}x_j\notin HN_{i,j}\cap KN_{i,j}.$ Let $N=\cap_{i,j}N_{i,j}.$ By construction, $x_i^{-1}x_j\notin HN\cap KN$ whenever $i\neq j$.
\end{proof}

% \begin{thm}
% 	Let $G$ be a  profinite group and let $P$ be a $p$-Sylow subgroup of $G.$  Suppose that $\mu(\nil_G(g))>0$ for every $g\in P.$
% 	Then the following hold:
% 	\begin{enumerate}
% 		\item $N_G(P)$ is an open subgroup of $G.$
% 		\item $P/O_p(G)$ is finite.
% 		\item $O_p(G)C_G(g)$ is open in $G$ for every $g \in P$.
% 		\item If $P$ is finite, then $C_G(P)$ is open in $G$.
% 	\end{enumerate}
% \end{thm}

\begin{proof}[Proof of \cref{localnilp}]
	Assume by contradiction that $H=N_G(P)$ is not open in $G.$ Then $H$ has infinite index in $G$, and therefore there exists a countable subset $\{g_n\}_{n\in \mathbb N}$ of $G$ such that $g_iH \neq g_jH$ whenever $i\neq j.$ Since $H$ is closed in $G$, for every pair $(i,j)$ with $i\neq j$ there exists an open normal subgroup $N_{i,j}$ of $G$ with $g_i^{-1}g_j\notin HN_{i,j}.$ Let $N=\cap_{i,j}N_{i,j}.$ Then $G/N$ is countably based, so by \cref{nilpel} $N_G(PN)$ has finite index in $G$. However $g_i^{-1}g_j$ does not normalize $PN$ if $i \neq j$, because $N_G(PN)=N_G(P)N=HN \subseteq HN_{i,j}$ where the first equality follows from the Frattini argument. This contradiction proves item (1). Now let $Q$ be another $p$-Sylow subgroup of $G$. Assume, by contradiction, that $|P:P\cap Q|$ is not finite. Then, by \cref{intersecochiusi}, there exists a closed normal subgroup $N$ of $G$ that $G/N$ is countably based and $|PN/N:PN/N\cap QN/N|=\infty,$ against \cref{nilpel} (2). Hence $|P: P\cap P^g|$ is finite for every $g\in G$. Since, by (1), $P$ has only finitely many conjugates in $G$, we deduce that $P/O_p(G)$ is finite, proving (2). Let $\Omega_p(G)$ be the set of $p$-elements of $G$. Since $|P/O_p(G)|$ and $|G:N_G(P)|$ are finite, there exists a finite subset $T$ of $G$ such that $\Omega_p(G) = \cup_{t\in T} tO_p(G).$ Now let $g\in \Omega_p(G).$ Since the Haar measure is invariant under conjugation, $\mu(\mathcal{N}_G(g))>0$. If $x\in \nil_G(g),$ then $x=x_1x_2$, where $x_1\in \Omega_p(G)$ and $[x_2,g]=1.$ Hence
    $$\nil_G(g)\subseteq \Omega_p(G)C_G(g)=\bigcup_{t\in T}tO_p(G)C_G(g).$$
	But then 
    \begin{align*}
    \eta & =\mu(\nil_G(g))\leq \mu(\cup_{t\in T}tO_p(G)C_G(g))=|T|\mu(O_p(G)C_G(g)) =\frac{|T|}{|G:O_p(G)C_G(g)|}
    \end{align*}
    and since $\eta > 0$, we conclude that 
	$|G:O_p(G)C_G(g)|^{-1}\geq \eta/|T|.$ This proves item (3). If $P$ is finite, then $O_p(G)$ is finite too hence $C_G(g)$ is open in $O_p(G) C_G(g)$ for every $g \in P$. Now item (3) implies that $C_G(g)$ is open in $G$ and, since $P$ is finite, $C_G(P) = \cap_{g \in P} C_G(g)$ is open too. This proves item (4).
\end{proof}

% \begin{prop}For every positive real number $\varepsilon$
% there exists a solvable profinite group $G$ which contains an element $g$ such that $\mu(\nil_G(g))>1-\varepsilon$,
% but $g$ has infinite order modulo $Z_\infty(G).$
% In particular, there is no open subgroup $H$ with $g\in Z_\infty(H)$.
% \end{prop}

Next, we show that \cref{nilpp} can't be generalized by considering the hypercenter instead of the Fitting subgroup.

\begin{proof}[Proof of \cref{nofiniteorder}]
Let $p$ and $q$ be two primes with $p$ dividing $q-1$ and consider the group $Y_t$ (associated to the primes $p$ and $q$) defined in the proof of \cref{allfingrps} (with the same notation) so that $|Y_t|=q^{n_t} p^{t+1}$ where $n_t=p^t$.  All the elements of $Y_t$ that are not contained in $N$ have order a $p$-power, and therefore $$\cfrac{|\Omega_p(Y_t)|}{|Y_t|} \geq \cfrac{|Y_t|-|N|}{|Y_t|} = 1-\cfrac{1}{p^{t+1}}.$$
 A $p$-Sylow subgroup $R$ of $Y_t$ has trivial normal core and $|Y_t:R|=q^{n_t}$. Thus, $Y_t$ embeds as a transitive subgroup into $\sym(q^{n_t})$ and we may consider the group $X_t = W_t \rtimes Y_t < C_{p^t} \wr Y_t$ where $C_{p^t}$ is cyclic of order $p^t$ and $W_t = \{(x_1,\ldots,x_{q^{n_t}}) \in (C_{p^t})^{q^{n_t}}\ |\ \prod_{i=1}^{q^{n_t}}x_i=1\}$. Then $X_t$ has trivial center, since the transitivity (and the faithfulness) of $Y_t$ implies that any element of $Z(X_t)$ has the form $(a,\ldots,a) \in (C_{p^t})^{q^{n_t}}$, and $a^{q^{n_t}} = \prod_{i=1}^{q^{n_t}}a = 1$ implies that $a=1$. Hence $\Z_\infty(X_t)=1.$ If we take any $g \in W_t \smallsetminus \{1\}$ and any $y \in \Omega_p(X_t)$, then $\la g,y\ra \leq \la O_p(X_t),y\ra$, which is a $p$-group, so $y \in \mathcal N_{X_t}(g)$. Notice also that every element in $W_t \Omega_p(Y_t)$ is a $p$-element and that
\[W_t \Omega_p(Y_t) = \bigcup_{b \in W_t} b\Omega_p(Y_t)\]
is a disjoint union. Therefore,
\[\cfrac{|\nil_{X_t}(g)|}{|X_t|} \geq \cfrac{|W_t||\Omega_p(Y_t)|}{|W_t||Y_t|}  \geq 1-\cfrac{1}{p^{t+1}}.\]
Now choose $s\in \mathbb N$ such that
$$\prod_{t \geq s}\left( 1- \cfrac{1}{p^{t+1}} \right) > 1-\varepsilon$$ and take the profinite group $G=\prod_{t\geq s} X_t.$ 
Clearly $Z_\infty(G)=\prod_{t\geq s} Z_\infty(X_t)=1$.
Now consider an element $g = (g_t)_{t \geq s}$ such that $g_t \in W_t$ has order $p^t$ for all $t$. Clearly $g$ has infinite order. On the other hand, $\nil_G(g)=\prod_{t\geq s}  \nil_G(g_t)$ and therefore
$$\mu(\nil_G(g)) \geq \prod_{t \geq s}\left( 1- \cfrac{1}{p^{t+1}} \right) > 1-\varepsilon.$$
For every $t\in \mathbb N,$ $X_t$ is solvable with derived length 4, so $G$ is solvable.

If $H$ is any open subgroup of $G$, then $g \not \in Z_{\infty}(H)$. Indeed, assume this is not true. There exists $r \geq s$ such that $N = \prod_{t \geq r} X_t \leq H$ and $N$ is an open normal subgroup of $G$. Now $\langle g \rangle \cap N \leq Z_{\infty}(H) \cap N \leq Z_{\infty}(N) = 1$, a contradiction since $g$ has infinite order.
\end{proof}

Finally, we prove a generalization of \cite[Theorem 29]{pel}.

\begin{proof}[Proof of \cref{classp}]Suppose $\mu(\mathcal P_p(G))>0$. Since $\mathcal P_p(G)\subseteq \mathcal S(G),$ it follows from \cref{descs} that $G$  is virtually prosolvable.
Let $C$ be the set of the elements $g$ in $G$ which centralize almost all the abelian chief factors of $G$ of order coprime to $p$, that is, there exists a positive integer $k$ such that, for every $N \unlhd_{\circ} G$, the set of abelian $p'$-chief factors of $G/N$ not centralized by $g$ has cardinality at most $k$. Then $C$ is a measurable subgroup of $G$. It follows from \cite[Corollary 24]{pel} that $\mathcal P_p(G)\subseteq C$, so $C$ has finite index and is closed in $G$. We claim that $C$ is virtually pro-$p.$ It is not restrictive to assume that $C$ is prosolvable and $O_p(C)=1.$ First, repeating the  arguments in \cite[Theorems 29 and 30]{pel} it can be proved that a $p$-Sylow subgroup $P$ of $C$ is finite. Since $P\leq C,$ every element of $P$ centralizes almost all the $p^\prime$-chief factors of $C$ and therefore there exists an open  normal subgroup $D$ of $C$ such that $P$ centralizes all  the $p^\prime$-chief factors of $D$. By \cite[Corollary 26]{pel} $P = O_p(PD)$, and therefore $D \subseteq N_C(P)$ so $|C:N_C(P)|$ is finite. But then $C$ contains only finitely many $p$-elements, so $\mathcal P_{p,C}(g)$
is finite for any $g \in C$. Thus, from the assumption $\mu(\mathcal P_p(G))>0$, we conclude that $C$, and hence $G$, is finite.
\end{proof}

%\textcolor{red}{Mi sembra che adesso le cose scritte nel paragrafo successivo sono già scritte nell'introduzione e quindi il paragrafo possa/debba essere tolto}

%We make a final remark. In \cite{aemp}, it is observed that the condition $\mathcal N(G) = G$ is stronger than $G$ being virtually pronilpotent; indeed, it suffices to consider the example $\mathbb Z_p \rtimes C_2 = \mathbb Z_p \rtimes \la x \ra$, where $p$ is an odd prime and $x$ acts on the $p$-adic integers by inversion, and observe that $\mathcal N_G(x) = \{1,x\}$. So it is natural to ask if $\mu(\mathcal N(G)) > 0$ if, and only if, $G$ is virtually pronilpotent. Compared to \cref{descs}, this is considerably harder to approach because we do not have a group-theoretical characterization of the {nilpotentizer} like we do for the {solvabilizer} nor do we know if $\mathcal N(G)$ is a subgroup. Hence, it would be desirable to find answers to \cref{q1} and \cref{q3} for the class of finite nilpotent groups.

\section{Measurable sets related to pronilpotency} \label{Lambda}

The proof of \cref{nobaer} relies on the following remark:

\begin{lemma}\label{sumcent}Let $G$ be a finite group, $p$ a prime, and let $N$ be a normal subgroup of $G$ whose order is coprime with $p.$ If $x$ is a $p$-element of $G$ and $y\in \Lambda_{G,p}(x)$, then
$\Lambda_{G,p}(x) \cap yN=yC_N(x^y)C_N(x)$.
\end{lemma}

\begin{proof}
Notice that $P:=\langle x, x^y\rangle$ is a $p$-Sylow subgroup of $PN.$ Now if $yn \in \Lambda_{G,p}(x) \cap yN,$ then $\langle x, x^{yn}\rangle=P^m$ for some $m\in N.$ If $h \in P$, then $|hN \cap P^m| \leq 1$ and moreover $h^m \in hN$, so $x^m=x$ and $x^{ym}=x^{yn}$. Therefore $m \in C_N(x)$ and
$mn^{-1}\in C_N(x^y).$ It follows that $yn \in \Lambda_{G,p}(x) \cap yN$ if and only if $n\in C_N(x^y)C_N(x)$.
\end{proof}

%\begin{proof}[Proof of \cref {nobaer}]
%Let $q$ be a prime such that $p$ divides $q-1$ and choose $n\in \mathbb N$ such that $n-1\geq \eta n.$
%${qn-q+1}\geq \eta qn.$ 
%Consider a cyclic
%group $A=\langle a\rangle$ of order $p$ and let $\mathbb F_q$ be the field with $q$ elements. Choose $\lambda \in \mathbb F_q^\times$ with $|\lambda|=p.$ Then
%$A$ acts on $\mathbb F_q$ by setting $f^a=\lambda f$ for every $f\in \mathbb F_q.$ Now let $\sigma=(1,2,\dots,n)\in \perm(n)$ and let $C=\langle \sigma \rangle \leq \perm(n).$ The action of $A$ on $\mathbb F_q$ gives rise to an action of the wreath product $H=A\wr C$ on the vector space $V=\mathbb F_q^n$ of dimension $n$ over $\mathbb F_q.$ We consider the semidirect product $G=V\rtimes H \leq {\rm{AGL}}(n,q).$ Since $\fit(G)=V,$ we have $O_p(G)=1.$ Let $x=(a,1,\dots,1)\in A^n\leq H.$
%Then $$C_V(x)=\{(0,f_2,\dots,f_n)\mid f_2,\dots,f_n\in V\}.$$
%For $1\leq i\leq n-1$ let $\Sigma_i:=A^n\sigma^i\ \subseteq H.$ If $g\in \Sigma_i$, we have
%$$C_V(x^g)=C_V(x^{\sigma^i})=\{(f_1,\dots,f_i,0,f_{i+2},\dots f_n)\mid f_1,\dots,f_i,f_{i+2},\dots,f_n\in V\}.$$
%In particular, $C_V(x)+C_V(x^{g})=V$. Since $g\in \Lambda_{G,p}(x)$, it follows from \cref{sumcent} that $Vg\subseteq \Lambda_{G,p}(x).$
%We have so proved that
%$\bigcup_{1\leq i\leq n-1}VA^n\sigma^i \subseteq \Lambda_{G,p}(x).$ 
%Moreover $C_G(x)\cong \mathbb F_q^{n-1}\langle x\rangle \subseteq \Lambda_{G,p}(x).$
%But then
%$$\frac{|\Lambda_{G,p}(x)|}{|G|}\geq\frac{n-1}{n}\geq \eta. \qedhere$$
%+\frac{1}{qn}=\frac{qn-q+1}{qn}\geq \eta.\qedhere $$
%\end{proof}

\begin{proof}[Proof of \cref{nobaer}]
Choose a prime $q$ such that $p^t$ divides $q-1$ and choose a positive integer $n$ such that $n-1\geq\eta n$. Let $A=\la a \ra$ be the cyclic group of order $p^t$,  let $\mathbb F_q$ be the field with $q$ elements and choose $\lambda \in \mathbb F_q^\times$ of order $p^t$. Then $A$ acts on $\mathbb F_q$ by $f^a = \lambda f$ for every $f \in \mathbb F_q$. Now, let $\sigma=(1,2,\dots,n)\in \perm(n)$
and let $C=\langle \sigma \rangle \leq \perm(n).$ The action of $A$ on $\mathbb F_q$ gives rise to an action of the wreath product $H=A\wr C$ on the vector space $V=\mathbb F_q^n$ of dimension $n$ over $\mathbb F_q.$ We consider the semidirect product $G=V\rtimes H \leq {\rm{AGL}}(n,q).$ {Since $H$ acts faithfully on $V$, it follows that $O_p(G)=1$.} Let $x=(a,1,\dots,1)\in A^n\leq H.$
Then $$C_V(x)=\{(0,f_2,\dots,f_n)\mid f_2,\dots,f_n\in V\}.$$
For $1\leq i\leq n-1$
let $\Sigma_i:=A^n\sigma^i\ \subseteq H.$ If $g\in \Sigma_i$, we have
$$C_V(x^g)=C_V(x^{\sigma^i})=\{(f_1,\dots,f_i,0,f_{i+2},\dots f_n)\mid f_1,\dots,f_i,f_{i+2},\dots,f_n\in V\}.$$
In particular, $C_V(x)+C_V(x^{g})=V$. Since $g\in \Lambda_{G,p}(x)$, it follows from \cref{sumcent} that $Vg\subseteq \Lambda_{G,p}(x).$
We have so proved that
$\bigcup_{1\leq i\leq n-1}VA^n\sigma^i \subseteq \Lambda_{G,p}(x).$ 
%Moreover $C_G(x)\cong \mathbb F_q^{n-1}\langle x\rangle \subseteq \Lambda_{G,p}(x).$
But then
$$\frac{|\Lambda_{G,p}(x)|}{|G|}\geq\frac{n-1}{n}\geq \eta. \qedhere$$
\end{proof}

As we have already recalled, by \cite[Theorem 1]{aemp} if $\mu(\nil_G(g))>0$ for every $g\in G$ then $G$ is virtually pronilpotent. This motivates the following question:
\begin{question}
	Let $G$ be a profinite group. Suppose that $\mu(\Lambda_{G}(g))>0$ for every  $g\in G.$ Is $G$ virtually pronilpotent?
\end{question}

We also propose the following related question.

\begin{question}
    Let $G$ be a profinite group and suppose that $$\mu(\{(x,g) \in G^2 \mid \la x,x^g\ra \text{ is pronilpotent}\})>0.$$ Is $G$ virtually pronilpotent?
\end{question}

	\end{document}